\documentclass[12pt,reqno]{amsart}
\usepackage{fullpage}
\usepackage{amsfonts}
\usepackage{amssymb}
\usepackage{enumerate}
\usepackage{times}
\usepackage{graphicx}
\usepackage{url}

\renewcommand{\mod}[1]{{\ifmmode\text{\rm\ (mod~$#1$)}\else\discretionary{}{}{\hbox{ }}\rm(mod~$#1$)\fi}}
\newcommand{\ord}{\mathop{\rm ord}\nolimits}
\newcommand{\ep}{\varepsilon}

\newcommand{\F}{{}_2F\!_1}
\newcommand{\N}{{\mathbb N}}

\newcommand{\R}{{\mathbb R}}
\newcommand{\Z}{{\mathbb Z}}

\newcommand{\dq}{Diophantine quadruple}
\newcommand{\dqs}{\dq s}
\newcommand{\dr}{doubly regular}
\newcommand{\drdq}{\dr\ \dq}
\newcommand{\drdqs}{\dr\ \dq s}

\newcommand{\var}{V}

\newtheorem{theorem}{Theorem}[section]
\newtheorem{proposition}[theorem]{Proposition}
\newtheorem{corollary}[theorem]{Corollary}
\newtheorem{lemma}[theorem]{Lemma}
{
\theoremstyle{remark}

}

\begin{document}

\title[Diophantine quadruples]{Erd\H os-Tur\'an with a moving target, equidistribution of roots of reducible quadratics, and Diophantine quadruples}
\author{Greg Martin and Scott Sitar}
\address{Department of Mathematics \\ University of British Columbia \\ Room
121, 1984 Mathematics Road \\ Canada V6T 1Z2}
\email{gerg@math.ubc.ca and sesitar@math.ubc.ca}
\maketitle

\begin{abstract}
A {\em Diophantine $m$-tuple} is a set $A$ of $m$ positive integers such that $ab+1$ is a perfect square for every pair $a,b$ of distinct elements of~$A$. We derive an asymptotic formula for the number of Diophantine quadruples whose elements are bounded by~$x$. In doing so, we extend two existing tools in ways that might be of independent interest. The Erd\H os-Tur\'an inequality bounds the discrepancy between the number of elements of a sequence that lie in a particular interval modulo~1 and the expected number; we establish a version of this inequality where the interval is allowed to vary. We also adapt an argument of Hooley on the equidistribution of solutions of polynomial congruences to handle reducible quadratic polynomials.
\end{abstract}

\thispagestyle{empty}
\section{Introduction}

A {\em Diophantine $m$-tuple} is a set $A$ of $m$ positive integers such that $ab+1$ is a perfect square for every pair $a,b$ of distinct elements of~$A$. For example, the first \dq\ $\{1,3,8,120\}$ was found by Fermat. (The nomenclature refers to Diophantus, who found a set that has the analogous property in the rational numbers, namely $\{1/16,33/16,17/4,105/16\}$.) It was proved by Dujella~\cite{other dujella} that there are no Diophantine sextuples and only finitely many Diophantine quintuples (while a folklore conjecture asserts that there are no Diophantine quintuples at all). However, there are infinitely many Diophantine $m$-tuples for $2\le m\le4$, and thus it is an interesting problem to try to count how many there are beneath a given bound.

Dujella showed~\cite{dujella} that the number of Diophantine pairs contained in $[1,x]$ is asymptotic to $\frac6{\pi^2}x\log x$, while the number of Diophantine triples contained in $[1,x]$ is asymptotic to $\frac3{\pi^2}x\log x$. He considered the same counting problem for \dqs, obtaining for sufficiently large $x$ the lower and upper bounds $0.1608 x^{1/3}\log x$ and $0.5354 x^{1/3}\log x$, respectively, for the number of \dqs\ contained in $[1,x]$. The primary purpose of this paper is to establish the following asymptotic formula:

\begin{theorem}
The number of \dqs\ contained in $[1,x]$ is given by the asymptotic formula
\[
  Cx^{1/3}\log x + O \big( x^{1/3} (\log x)^{2/3+\sqrt{2}/6} (\log \log x)^{5/12} \big),
\]
where $C ={2^{4/3}/3\Gamma(\tfrac23)^3} \approx 0.338285$.
\label{main thm}
\end{theorem}

\noindent Note that the exponent of $\log x$ in the error term, $\frac23+\frac{\sqrt2}6 \approx 0.90237$, is indeed slightly smaller than~1.

In the course of establishing Theorem~\ref{main thm}, we found ourselves needing to extend two existing tools from the literature to suit our needs; these extensions might be of interest in their own right. The first of these tools is the Erd\H os-Tur\'an inequality, which gives a quantitative bound for the discrepancy between the number of elements of a sequence that lie in a particular interval modulo~1 and the expected number. We require a version of this inequality in which the target interval is allowed to vary. To set some notation, let $u=\{ u_n \}$, $\alpha=\{\alpha_n\}$, and $\beta=\{\beta_n\}$ be sequences of real numbers; we are interested in counting how many elements $u_n$ lie in the corresponding interval $[\alpha_n,\beta_n]$ modulo~1. (So that these intervals modulo~1 are sensible, we make the restriction $\alpha_n \leq \beta_n \leq \alpha_n + 1$ for all~$n$.) Define the counting function
\begin{equation}
  Z_N = Z_N(u;\alpha,\beta) = \#\{ 1\le n \le N : u_n \in [ \alpha_n, \beta_n ] \mod{1} \}
\label {ZN def}
\end{equation}
and the discrepancy between the counting function and the expected number
\begin{equation}
  D_N = D_N(u;\alpha,\beta) = Z_N - \sum_{n=1}^N (\beta_n - \alpha_n).
\label {DN def}
\end{equation}
For any sequence $\{s_n\}$, let $\var_N(s) = \sum_{n=1}^{N-1} \|s_{n+1}-s_n\|$ denote its total variation considered as a sequence modulo~1, where $\|y\| = \min_{z\in\Z} |y-z|$ is the natural metric on $\R/\Z$. We prove the following ``moving target'' extension of the Erd\H os--Tur\'an inequality in Section~\ref{ET section}.

\begin{theorem}
Let $\{u_n\}$, $\{\alpha_n\}$, and $\{\beta_n\}$ be sequences of real numbers with $\alpha_n \leq \beta_n \leq \alpha_n + 1$. With $D_N = D_N(u;\alpha,\beta)$ as defined as in equation~\eqref{DN def}, we have
\begin{equation}
  |D_N| \le \frac{N}{H+1} + \sum_{h=1}^H \big( 1 + \pi h ( \var_N(\alpha) + \var_N(\beta)) \big) \bigg( \frac{2-c}{H+1} + \frac{c}{h} \bigg) M_N(h)
\label{refer to constants}
\end{equation}
for any positive integers $N$ and $H$, where $c=16/7\pi$ and
\begin{equation}
  M_N(h) = \max_{1\le T\le N} \bigg| \sum_{n=1}^T e(hu_n) \bigg|.
\label{MNh def}
\end{equation}
\label{moving target thm}
\end{theorem}

\noindent The following corollary is an immediate consequence of Theorem~\ref{moving target thm} and is easy to apply in many concrete situations.

\begin{corollary}
Let $\{u_n\}$, $\{\alpha_n\}$, and $\{\beta_n\}$ be sequences of real numbers with $\alpha_n \leq \beta_n \leq \alpha_n + 1$. Suppose that both $\alpha$ and $\beta$ are monotone sequences. With $D_N = D_N(u;\alpha,\beta)$ as defined as in equation~\eqref{DN def}, we have
\begin{equation*}
  |D_N| \ll \frac{N}{H} + \big( 1 + |\alpha_N-\alpha_1| + |\beta_N-\beta_1| \big) \sum_{h=1}^H M_N(h)
\end{equation*}
for any positive integers $N$ and $H$, where $M_N(h)$ is defined in equation~\eqref{MNh def}.
\label{moving target cor}
\end{corollary}

\noindent We compare our Theorem~\ref {moving target thm} to the standard Erd\H os--Tur\'an inequality, and exhibit two examples that probe the sharpness of our inequality, at the end of Section~\ref{ET section}.

The second tool that we extend is a result of Hooley on the equidistribution of the roots of polynomial congruences. Specifically, given any polynomial $f(t) \in \Z[t]$, we consider the sequence of real numbers $\frac\nu k$, where $k$ runs over all positive integers up to some bound $y$ and $\nu$ runs, for each $k$, over the roots of the congruence $f(\nu) \equiv 0 \mod{k}$. Corresponding to such a sequence, we define the exponential sum
\begin{equation}
  R_f(h,y) = \sum_{k \leq y} \sum_{\substack{f(\nu) \equiv 0 \mod{k} \\ 0 < \nu \leq k}} e^{2\pi i h \nu/k}.
\label{Rfhx definition}
\end{equation}
Hooley~\cite{hooley, other hooley} established nontrivial upper bounds for $R_f(h,y)$ when $f$ is an irreducible polynomial of degree at least~2. By adapting the methods of~\cite{hooley}, we prove the following analogous upper bound for reducible quadratic polynomials in Section~\ref {equidist section}.

\begin{theorem}
Let $f$ be a reducible, nonsquare quadratic polynomial with integer coefficients, and let $D$ be its discriminant. Let $R_f(h,y)$ be defined as in equation~\eqref{Rfhx definition}. For all real numbers $y\ge3$ and for every integer $h \ne 0$,
$$
  R_f(h,y) \ll \sqrt D \prod_{p\mid h} \bigg( 1+\frac7{\sqrt p} \bigg) \cdot y (\log y)^{\sqrt{2}-1} (\log \log y)^{5/2}.
$$
\label{equidist thm}
\end{theorem}

The trivial bound for $R_f(h,x)$, namely the number of summands, has order of magnitude $y\log y$ (see Lemma~\ref {rho asymptotic lemma}), and so Theorem~\ref{equidist thm} represents a nontrivial upper bound for the sum when $y$ is sufficiently large. It follows immediately from Weyl's criterion~\cite[page 1]{montgomery} that the normalized roots $\nu/k$ are equidistributed modulo~1. We remark that for reducible quadratic polynomials $f$, the true order of magnitude of $R_f(h,y)$ is probably $y$; therefore the estimate in Theorem~\ref{equidist thm} cannot be improved too much. We discuss these issues further in Section~\ref{equidist section}.

In Section~\ref{main proof section} we show how these two tools can be used to prove Theorem~\ref{main thm}. Several times in Sections~\ref{equidist section} and~\ref{main proof section}, we need to invoke standard results on sums of multiplicative functions; to preserve the flow of ideas, we defer the proofs of all such results to Section~\ref{mult fn section}.

\section{Erd\H os-Tur\'an with a moving target}
\label{ET section}

In this section we prove Theorem~\ref{moving target thm}. We begin by recalling the standard approach to bounding the discrepancy $D_N$ using exponential sums and Selberg's ``magic functions''. We then derive a bound for the individual Fourier coefficients that arise when using Selberg's functions. At that point, we can use partial summation to finish the proof of Theorem~\ref{moving target thm}. At the end of the section, we compare our ``moving target'' inequality to the standard Erd\H os--Tur\'an inequality and exhibit two examples that probe the sharpness of our inequality.

\subsection{Bounding the discrepancy by exponential sums}

We begin by defining, for any positive integer $H$, the trigonometric polynomial of degree~$H$
\begin{equation}
  B_H(x) = \frac{1}{H+1} \sum_{h=1}^H f \bigg( \frac{h}{H+1} \bigg) \sin 2\pi h x + \frac1{2(H+1)}\sum_{h=-H}^H \bigg( 1 - \frac{|h|} {H+1} \bigg) e(hx),
\label{BH def}
\end{equation}
where $e(y) = e^{2\pi iy}$ as usual and
\begin{equation}
  f(y) = -(1-y) \cot \pi y - \tfrac{1}{\pi}.
\label{f def}
\end{equation}
We have $\hat{B}_H(0) = {1/(2(H+1))}$ and $\hat B_H(h) = 0$ if $|h|\ge H+1$, while if $1\le h\le H$ then
\begin{equation}
  \hat{B}_H(\pm h) = \frac{1}{2(H+1)} \bigg( 1 - \frac{|h|}{H+1}
    \mp i f \bigg( \frac{|h|}{H+1} \bigg) \bigg).
\label{BH hat formula}
\end{equation}
We then define, for any real numbers $\alpha$ and $\beta$ satisfying $\alpha\le \beta\le\alpha+1$, the two related trigonometric polynomials
\begin{align*}
  S_H^+(\alpha,\beta; y) & = \beta - \alpha +
    B_H(y-\beta) + B_H(\alpha - y) \\
  S_H^-(\alpha,\beta; y) & = \beta - \alpha -
    B_H(\beta-y) - B_H( y-\alpha ),
\end{align*}
which satisfy $\hat{S}_H^\pm(0) = \beta-\alpha\pm1/(H+1)$ and, for $h\ne0$,
\begin{align}
  \hat{S}_H^+(h) &= \hat{B}_H(h) e(- h\beta) + \hat{B}_H(- h) e(h\alpha) \notag \\
  \hat{S}_H^{-}(h) &= -\hat{B}_H(- h) e(h\beta) - \hat{B}_H(h) e(- h\alpha).
\label{S Fourier coefficients}
\end{align}
These trigonometric polynomials are useful one-sided approximants to the characteristic function $\chi(\alpha,\beta; y)$ of the interval $[\alpha,\beta]$ modulo~1, which equals~1 if there is some number $z$ between $\alpha$ and $\beta$ such that $y\equiv z\mod 1$ and~0 otherwise:

\begin{proposition}
For any real numbers $\alpha$ and $\beta$ satisfying $\alpha\le \beta\le\alpha+1$,
$$
  S_H^-(\alpha,\beta; y) \leq \chi(\alpha,\beta; y) \le S_H^+(\alpha,\beta; y)
$$
for all real numbers~$y$.
\label{magic functions prop}
\end{proposition}

\begin{proof}
This is the fundamental property of the ``Selberg magic functions'' $S_H^\pm$; see \cite[chapter 1]{montgomery} for an exposition which uses the notation we have employed.
\end{proof}

Note that the definition~\eqref{ZN def} of $Z_N$ can be written as $Z_N = \sum_{n\le N} \chi(\alpha_n,\beta_n ; u_n)$. Following the approach to proving the standard Erd\H os--Tur\'an inequality, we use Proposition~\ref{magic functions prop} to write the upper bound
\begin{equation*}
  Z_N \le \sum_{n\le N} S_H^+ (\alpha_n,\beta_n ; u_n) = \sum_{n\le N} \bigg( \beta_n - \alpha_n + \frac1{H+1} + \sum_{1\le |h|\le H} \hat S_H^+(h) e(hu_n) \bigg)
\end{equation*}
(where we have singled out the constant term $\hat S_H^+(0)$ of $S_H^+$); the definition~\eqref{DN def} of $D_N$ thus yields
\begin{align*}
  D_N &\le \frac N{H+1} + \sum_{n\le N} \sum_{1\le |h|\le H} \hat S_H^+(h) e(hu_n) \\
  &= \frac N{H+1} + \sum_{n\le N} \sum_{1\le |h|\le H} \big( \hat{B}_H(h) e(-h \beta_n) + \hat{B}_H(-h) e(h\alpha_n) \big) e(h u_n)
\end{align*}
by equation~\eqref{S Fourier coefficients}. Interchanging the order of summation, we get
\begin{equation}
  D_N \le \frac N{H+1} + \sum_{1 \leq |h| \leq H} \bigg( \hat{B}_H(h) \sum_{n\le N} e(h u_n) e(-h \beta_n) + \hat{B}_H(-h) \sum_{n\le N} e(h u_n) e(h \alpha_n) \bigg).
\label{DN upper bound}
\end{equation}
The same calculation using $S_H^-$ instead of $S_H^+$ yields the corresponding lower bound
\begin{equation}
  D_N \ge -\frac N{H+1} - \sum_{1 \leq |h| \leq H} \bigg( \hat{B}_H(-h) \sum_{n\le N} e(h u_n) e(h \beta_n) + \hat{B}_H(h) \sum_{n\le N} e(h u_n) e(-h \alpha_n) \bigg).
\label{DN lower bound}
\end{equation}

If the sequences $\alpha$ and $\beta$ were constant, as in the standard Erd\H os--Tur\'an inequality, then we could now factor out the exponential sum $\sum_{n\le N} e(hu_n)$ and then bound the remaining sum over $h$ at once. Instead, we bound the Fourier coefficients $\hat B_H(h)$ individually and use partial summation to estimate the effect of the varying terms $e(-h\beta_n)$ and $e(-h\alpha_n)$ on the exponential sums.

\subsection{Bounding the Fourier coefficients}

We begin by establishing an inequality between the cotangent function and a rational function that we will use to bound the function $f$ from the previous section.

\begin{lemma}
For $0<y<1$, we have
\[
\pi\cot\pi y + \frac1{1-y} < \frac1y + \frac{3(1-y)}2 - \frac1{2-y}.
\]
\label{little cot lemma}
\end{lemma}

\begin{proof}
We start with the classical equality~\cite[equation (4.3.91)]{AS}
\begin{align*}
\pi\cot\pi z &= \frac1z + \sum_{\substack{n\in\Z \\ n\ne0}} \bigg( \frac1{z+n}-\frac1n \bigg) \\
&= \frac1z + \frac1{z+1} + \frac1{z-1} - 2z \sum_{n=2}^\infty \frac1{n^2-z^2} \\
&> \frac1z + \frac1{z+1} + \frac1{z-1} - 2z \sum_{n=2}^\infty \frac1{n^2-1} = \frac1z + \frac1{z+1} + \frac1{z-1} - \frac{3z}2,
\end{align*}
where the inequality is valid for $|z|<1$. Substituting $z=1-y$ yields
\[
-\pi\cot\pi y = \pi\cot(\pi(1-y)) >  \frac1{1-y} + \frac1{2-y} - \frac1y - \frac{3(1-y)}2,
\]
which is equivalent to the statement of the lemma.
\end{proof}

\begin{lemma}
Let $f$ be defined as in equation~\eqref{f def}. For $0<y<1$, we have $|f(y)| < \frac c2\big( \frac1y-1 \big)$, where $c=\frac{16}{7\pi}$.
\label{f bound lemma}
\end{lemma}

\begin{proof}
It is equivalent to show that
\[
\frac{\pi|f(y)|}{1-y} - \frac8{7y} < 0
\]
for $0<y<1$. By Lemma~\ref{little cot lemma}, we have
\begin{align*}
\frac{\pi|f(y)|}{1-y} - \frac8{7y} &= \bigg( \pi\cot\pi y + \frac1{1-y} \bigg) - \frac8{7y} \\
&< \bigg( \frac1y + \frac{3(1-y)}2 - \frac1{2-y} \bigg) - \frac8{7y} = \frac{21 y^3-63 y^2+30 y-4}{14y(2-y)}.
\end{align*}
The denominator of the right-hand side is obviously positive for $0<y<1$; it suffices to show that the numerator is negative in that range. The polynomial $p(t) = 21 t^3-63 t^2+30 t-4$ has negative discriminant ${-29}$,484, and so it has exactly one real root. Moreover, $p(1) = -16$ and $\lim_{y\to\infty} p(y) = \infty$, so that real root is greater than~1; in particular, $p(y)<0$ for all $y<1$.
\end{proof}

We are now able to bound the Fourier coefficients $\hat B_H(h)$.

\begin{lemma}
Let $B_H$ be defined as in equation~\eqref{BH def}. For $h \neq 0$, we have
$$
  \big\vert \hat{B}_H(h) \big\vert < \frac14 \bigg( \frac{2-c}{H+1} + \frac{c}{|h|} \bigg),
$$
where $c = \frac{16}{7\pi}$.
\label{BHhat bound lemma}
\end{lemma}

\begin{proof}
From the formula~\eqref{BH hat formula} for the Fourier coefficients of $B_H$, we see that
\begin{align}
  \big\vert \hat{B}_H(h) \big\vert & =
    \frac{1}{2(H+1)} \bigg( \bigg( 1 - \frac{|h|}{H+1} \bigg)^2
    + f \bigg( \frac{|h|}{H+1} \bigg)^2 \bigg)^{1/2} \notag \\
  & \le \frac{1}{2(H+1)} \bigg( 1 - \frac{|h|}{H+1}
    + \bigg\vert f \bigg( \frac{|h|}{H+1} \bigg) \bigg\vert \bigg) \label{two inequalities} \\
  & < \frac{1}{2(H+1)} \bigg( 1 + \frac c2 \bigg(
    \frac{1 - |h|/(H+1)}{|h|/(H+1)} \bigg) \bigg) \notag
\end{align}
by Lemma~\ref{f bound lemma}; this inequality is equivalent to the statement of the lemma.
\end{proof}

\subsection{Finishing the proof of Theorem~\ref{moving target thm}}

At this point, we need only to record the outcome of a partial summation argument to be fully prepared to prove Theorem~\ref{moving target thm}. Recall that $\var_N(s) = \sum_{n=1}^{N-1} \|s_{n+1}-s_n\|$ denotes the total variation modulo~1 of the sequence $s=\{s_n\}$.

\begin{lemma}
For any sequences $\{u_n\}$ and $\{s_n\}$ of real numbers and for any integer $h$,
$$
  \bigg\vert \sum_{n=1}^N e(h u_n) e(-h s_n) \bigg\vert \le \big( 1 + 2\pi |h| \var_N(s) \big) M_N(h),
$$
where $M_N(h)$ is defined in equation~\eqref{MNh def}.
\label{partial summation lemma}
\end{lemma}

\begin{proof}
Let $E_T(h) = \sum_{n=1}^T e(h u_n)$. Using partial summation, we have
\begin{align*}
  \sum_{n=1}^N e(h u_n) e(-hs_n) & =
    \sum_{n=1}^N \big( E_n(h) - E_{n-1}(h) \big) e(-hs_n) \\
  & = E_N(h) e(-hs_N) - \sum_{n=1}^{N-1} E_n(h)
    ( e(-hs_{n+1}) - e(-hs_n) ),
\end{align*}
and so by the triangle inequality
\begin{align*}
  \bigg\vert \sum_{n=1}^{N-1} e(h u_n) e(-hs_n) \bigg\vert &\le |E_N(h)| + \sum_{n=1}^ {N-1} |E_n(h)| \cdot | e(-hs_{n+1}) - e(-hs_n) | \\
  &\le M_N(h) \bigg( 1 + \sum_{n=1}^{N-1} | e(-hs_{n+1}) - e(-hs_n) | \bigg).
\end{align*}
From elementary properties of the exponential function,
\[
| e(-hs_{n+1}) - e(-hs_n) | = \big| 1 - e\big(h(s_{n+1}-s_n)\big) \big| = \big| 1 - e\big(h\|s_{n+1}-s_n\|\big) \big|;
\]
furthermore, $| 1 - e(y)| \le 2\pi|y|$ by the mean value theorem. We conclude that
\begin{equation*}
  \bigg\vert \sum_{n=1}^{N-1} e(h u_n) e(-hs_n) \bigg\vert \le M_N(h) \bigg( 1 + 2 \pi |h| \sum_{n=1}^{N-1} \| s_{n+1} - s_n \| \bigg),
\end{equation*}
which establishes the lemma.
\end{proof}

\begin{proof}[Proof of Theorem~\ref{moving target thm}]
%We note that  $M_N(h)$ is always positive and $M_N(-h)=M_N(h)$.
Beginning with the upper bound~\eqref{DN upper bound} on $D_N$, 
\begin{equation*}
  D_N \le \frac N{H+1} + \sum_{1 \leq |h| \leq H} \bigg( \hat{B}_H(h) \sum_{n\le N} e(h u_n) e(-h \beta_n) + \hat{B}_H(-h) \sum_{n\le N} e(h u_n) e(h \alpha_n) \bigg),
\end{equation*}
we use Lemma~\ref{BHhat bound lemma} to obtain
\begin{equation*}
  D_N \le \frac N{H+1} + \sum_{1 \leq |h| \leq H} \frac14 \bigg( \frac{2-c}{H+1} + \frac{c}{|h|} \bigg) \bigg( \bigg| \sum_{n\le N} e(h u_n) e(-h \beta_n) \bigg| + \bigg| \sum_{n\le N} e(h u_n) e(h \alpha_n) \bigg| \bigg).
\end{equation*}
where $c = \frac{16}{7\pi}$. We now invoke Lemma~\ref{partial summation lemma} to see that
\begin{align*}
  D_N &\le \frac N{H+1} \\
  &\qquad{}+ \sum_{1 \leq |h| \leq H} \frac14 \bigg( \frac{2-c}{H+1} + \frac{c}{|h|} \bigg) \bigg( \big( 1 + 2\pi |h| \var_N(\beta) \big) M_N(-h) + \big( 1 + 2\pi |h| \var_N(\alpha) \big) M_N(h) \bigg) \\
  &\le \frac N{H+1} + \frac14 \sum_{1 \leq |h| \leq H} \bigg( \frac{2-c}{H+1} + \frac{c}{|h|} \bigg) \big( 2 + 2\pi |h| ( \var_N(\alpha) + \var_N(\beta) ) \big) M_N(h),
\end{align*}
since $M_N(-h)=M_N(h)$. At this point the summands for $h$ and $-h$ are equal, and so
\[
  D_N \le \frac N{H+1} + \sum_{h=1}^H \bigg( \frac{2-c}{H+1} + \frac{c}{h} \bigg) \big( 1 + \pi h ( \var_N(\alpha) + \var_N(\beta) ) \big) M_N(h).
\]
This is precisely the upper bound claimed in the statement of Theorem~\ref{moving target thm}; the lower bound is established by exactly the same proof, starting with the lower bound~\eqref{DN lower bound} on $D_N$.
\end{proof}

\subsection{Probing the sharpness of Theorem~\ref{moving target thm}}

The standard Erd\H os--Tur\'an inequality~\cite[equation~(23)]{montgomery} has the form
\[
  |D_N| \le \frac N{H+1} + \sum_{h=1}^H \bigg( \frac{2}{H+1} + \min\Big\{ \beta-\alpha, \frac{2/\pi}{h} \Big\} \bigg) \bigg| \sum_{n=1}^N e(hu_n) \bigg|.
\]
If we restrict $\alpha$ and $\beta$ to be constant sequences in Theorem~\ref{moving target thm}, the conclusion is that
\begin{align*}
  |D_N| &\le \frac N{H+1} + \sum_{h=1}^H \bigg( \frac{2-c}{H+1} + \frac{c}{h} \bigg) M_N(h) \\
  &= \frac N{H+1} + \sum_{h=1}^H \bigg( \frac{2-c}{H+1} + \frac{c}{h} \bigg) \max_{1\le T\le N} \bigg| \sum_{n=1}^T e(hu_n) \bigg|.
\end{align*}
In theory, this conclusion is weaker than the traditional Erd\H os--Tur\'an inequality due to the presence of the maximum; in practice, however, the attainable bounds on the exponential sums that arise are increasing functions of $T$, and so nothing would be lost. (Our method also does not include the possibility of replacing the $\frac ch$ with $\beta-\alpha$, as the latter difference is not independent of $n$.)

Moreover, in its full ``moving target'' generality, the term $M_N(h)$ in Theorem~\ref{moving target thm} cannot be replaced by $\big| \sum_{n=1}^N e(hu_n) \big|$. To see this, we consider the sequence $\{u_n\} = \{n/(N+1)\}$ for some positive integer $N$, and we select the ``obliging target'' intervals bounded by the sequences
\begin{equation}
  \{\alpha_n\} = \{ u_n - 2^{-n} \} \quad\text{and}\quad \{\beta_n\} = \{ u_n + 2^{-n} \}.
\label{obliging}
\end{equation}
The total lengths of these intervals is bounded by~2, but the number of points in the sequence $u$ that lie in the intervals $[\alpha,\beta]$ is~$N$; therefore $D_N \ge N-2$ in this situation. However, the total variations of the sequences $\alpha$ and $\beta$ are bounded by $\frac32$, while the exponential sum $\sum_{n=1}^N e(hu_n) = \sum_{n=1}^N e(hn/(N+1))$ has the exact value of~$-1$ for every integer $1\le h\le N$. Therefore, for any integer $H$ between~1 and~$N$,
\begin{align*}
  \frac{N}{H+1} + \sum_{h=1}^H \big( 1 + \pi h ( \var_N(\alpha) + \var_N(\beta)) \big) & \bigg( \frac{2-c}{H+1} + \frac{c}{h} \bigg) \bigg| \sum_{n=1}^N e(hu_n) \bigg| \\
  &< \frac N{H+1} + \sum_{h=1}^H ( 1 + 3\pi h) \bigg( \frac{2-c}{H+1} + \frac{c}{h} \bigg) \ll \frac NH + H.
\end{align*}
Since the right-hand side can be substantially smaller than $N-2$ (when $H$ is near $\sqrt N$, for instance), it is impossible for the left-hand side to be an upper bound for $D_N$. Theorem~\ref{moving target thm} remains valid in this situation, though, since $M_N(h)$ is approximately $\frac Nh$, whence the right-hand side of equation~\eqref{refer to constants} has order of magnitude $N\log H$.

We can also consider the sequence $\{u_n\} = \{n^\gamma\}$ for some real number $0<\gamma<1$ and the same ``obliging intervals''~\eqref{obliging} as before, so that the discrepancy is again at least $N-2$. The terms on the right-hand side of equation~\eqref{refer to constants} must therefore conspire to give a contribution whose order of magnitude is at least~$N$. The total variations of $\alpha$ and $\beta$ are each $N^\gamma+O(1)$, while the exponential sum $\sum_{n=1}^N e(hn^\gamma)$ can be shown to be asymptotic to $N^{1-\gamma} e(hN^\gamma)/(2\pi i h \gamma)$, which has order of magnitude $N^{1-\gamma}/h$. Therefore the right-hand side of equation~\eqref{refer to constants} has order of magnitude
\[
  \frac NH + \sum_{h=1}^H N^\gamma \frac{N^{1-\gamma}/h}h \sim N\log H,
\]
showing that the bound is both correct and reasonably tight in this case as well.

Both of these examples provide extremely positive discrepancies $D_N$, but simply replacing the ``obliging intervals'' $[\alpha_n,\beta_n]$ with their complements $[\beta_n,\alpha_n+1]$ yields extremely negative disrepancies instead.

Different methods of bounding the Fourier coefficients $\hat B_H(h)$ yield different constants, in equation~\eqref{refer to constants}, in the numerators of the fractions whose denominators are $H+1$ and $h$. The constants $2-c$ and $c$ we have derived above compare favorably to the constants yielded by other methods. Still, we can immediately identify two ways one could improve these constants if such an improvement were desired. First, the constant $\frac c2$ in Lemma~\ref {f bound lemma} is approximately 0.363783; from a numerical study, the best possible constant in this part of the argument would be approximately 0.356113, a difference of about 2\%. Even aside from this, both of the inequalities in equation~\eqref{two inequalities} can be noticeably improved for $|h|$ near~$H$ by elementary means.

\section{Equidistribution of roots of quadratic congruences}
\label{equidist section}

Our goal here is to modify the argument in Hooley's paper to show that the roots of a reducible quadratic are equidistributed in the same sense as his paper.

\subsection{Number of roots of reducible quadratics}

For any polynomial $f$ with integer coefficients, define
\[
\rho_f(m) = \#\{1\le r\le m\colon f(r)\equiv0\mod m\}
\]
to be the number of roots of $f\mod m$ in a block of $m$ consecutive integers. By the Chinese remainder theorem, the function $\rho_f$ is multiplicative for any~$f$. We also use the common notation $\ord_p(n)$ to denote the multiplicity with which the prime $p$ divides $n$, and we write $p^\alpha \parallel n$ when $\ord_p(n)=\alpha$. 

Recall that the {\em content} of a polynomial with integer coefficients is the greatest common divisor of its coefficients. A polynomial with integer coefficients is called {\em primitive} if its content equals~1. Our first lemma allows us to reduce the task of counting roots of polynomials$\mod m$ to the primitive case.

\begin{lemma}
Given $f(t)\in\Z[t]$, let $W$ be the content of $f$ and write $g(t)=\frac1Wf(t)\in\Z[t]$. Let $p$ be a prime, and let $\gamma=\ord_p(W)$. Then for any positive integer $\alpha$,
\[
\rho_f(p^\alpha) = \begin{cases}
p^\alpha, & \text{if } \alpha\le\gamma, \\
p^\gamma \rho_g(p^{\alpha-\gamma}), & \text{if } \alpha>\gamma.
\end{cases}
\]
\label{evict content lemma}
\end{lemma}

\begin{proof}
The congruence $f(r) = Wg(r) \equiv 0\mod m$ is equivalent~\cite[Theorem 2.3(1)]{NZM} to $g(r) \equiv 0\mod{\frac m{(W,m)}}$. When $m=p^\alpha$ with $\alpha\le\gamma$, then $(W,m) = p^\alpha$ and the congruence is equivalent to $g(r)\equiv0\mod1$; this is satisfied by every integer, in particular by all $p^\alpha$ of the integers between 1 and $p^\alpha$. When $m=p^\alpha$ with $\alpha>\gamma$, then $(W,m) = p^\gamma$ and the congruence is equivalent to $g(r)\equiv0\mod{p^{\alpha-\gamma}}$; this is satisfied by $\rho_g(p^{\alpha-\gamma})$ integers in every block of $p^{\alpha-\gamma}$ consecutive integers, in particular by $p^\gamma \rho_g(p^{\alpha-\gamma})$ of the integers between 1 and $p^\alpha$.
\end{proof}

In our investigation of the values of $\rho_g$ for primitive reducible quadratics $g$, we will need the following elementary but awkward result.

\begin{lemma}
Fix integers $\alpha$ and $\delta$. Suppose that $y$ and $z$ are integers satisfying either
\begin{enumerate}
\renewcommand{\labelenumi}{(\roman{enumi})}
  \item $y=z < \delta$ or
  \item $\min\{y,z\} = \delta$.
\end{enumerate}
If $\alpha\le2\delta$, then the inequality $y+z\ge\alpha$ is equivalent to $y \ge \lceil \frac\alpha2 \rceil$, while if $\alpha>2\delta$, then the inequality $y+z\ge\alpha$ is equivalent to the conjunction

\smallskip
\centerline{$
\min\{y,z\} = \delta \text{ and } \max\{y,z\}\ge \alpha-\delta.
$}
\label{isolate the annoyance lemma}
\end{lemma}

\begin{proof}
We begin by assuming that $\alpha\le2\delta$. First suppose that $y \ge \lceil \frac\alpha2 \rceil$, which is equivalent to $y \ge \frac\alpha2$ since $y$ is an integer. If (i) holds, then indeed $y+z = 2y \ge \alpha$; while if (ii) holds, then $y+z \ge y+\delta \ge \frac\alpha2+\frac\alpha2 = \alpha$. On the other hand, suppose that $y < \lceil \frac\alpha2 \rceil$, which is equivalent to $y < \frac\alpha2$. If (i) holds, then $y+z=2y<\alpha$; while (ii) is impossible due to the contradiction $y<\frac\alpha2\le\delta = \min\{y,z\} \le y$. This establishes the first asserted equivalence.

Now we assume that $\alpha>2\delta$. The condition (i) forces both $y+z<2\delta<\alpha$ and $\min\{y,z\} < \delta$, which makes both sides of the second asserted equivalence false. On the other hand, under condition (ii) the inequalities $y+z\ge\alpha$ and $\max\{y,z\}\ge \alpha-\delta$ are equivalent, since $y+z=\min\{y,z\}+\max\{y,z\}$. This establishes the second asserted equivalence.
\end{proof}

Any primitive reducible quadratic can be written as $g(t) = (at+b)(ct+d)$, where the primitivity implies that $(a,b)=(c,d)=1$. Define $\Delta=|ad-bc|$, and note that the discriminant of $g$ equals~$\Delta^2$. We note that $g$ is the square of a linear polynomial if and only if $\Delta=0$.

\begin{lemma}
Let $g(t) = (at+b)(ct+d)$ where $a,b,c,d\in\Z$ and $(a,b)=(c,d)=1$, and define $\Delta=|ad-bc|$. Assume that $\Delta\ne0$. Let $p$ be a prime, and let $\delta=\ord_p(\Delta)$. Then for any positive integer $\alpha$,
\begin{equation}
\rho_g(p^\alpha) = \begin{cases}
2, & \text{if } p\nmid ac\Delta, \\
0, & \text{if } p\mid ac \text{ and } p\mid\Delta, \\
1, & \text{if } p\mid ac \text{ and } p\nmid\Delta, \\
p^{\lfloor \alpha/2 \rfloor}, & \text{if } p\nmid ac \text{ and } p\mid\Delta \text{ and } \alpha\le 2\delta, \\
2p^\delta, & \text{if } p\nmid ac \text{ and } p\mid\Delta \text{ and } \alpha>2\delta.
\end{cases}
\label{primitive rho formula}
\end{equation}
In particular, $\rho_g(p^\alpha) \le 2p^\delta$.
\label{primitive rho lemma}
\end{lemma}

We remark that since $(a,b)=(c,d)=1$, any prime that divides two of $a$, $c$, and $\Delta$ automatically divides the third. Therefore the second line of the formula~\eqref{primitive rho formula} is the case where $p$ divides both of $a$ and $c$, while the third line is the case where $p$ divides exactly one of $a$ and~$c$.

\begin{proof}
First suppose that $p$ divides both $a$ and $c$. Then the congruence $g(r)\equiv0\mod{p^\alpha}$ implies $bd\equiv0\mod p$, which is a contradiction since $(a,b)=(c,d)=1$. Therefore $\rho_g(p^\alpha)=0$ in this~case.

Next suppose that $p\nmid\Delta$. Since $\Delta$ is an integer linear combination of $ar+b$ and $cr+d$ for any integer $r$, it is impossible for $p$ to divide both such integers simultaneously. Thus the number of roots of $g(r)\equiv0\mod{p^\alpha}$ is simply the sum of the numbers of roots of $ar+b\equiv0\mod{p^\alpha}$ and $cr+d\equiv0\mod{p^\alpha}$. The number of roots modulo $p^\alpha$ of $ar+b\equiv0\mod{p^\alpha}$ is 1 if $p\nmid a$ and 0 if $p\mid a$ (since $p\nmid b$ in the second case), and similarly for the number of roots modulo $p^\alpha$ of $cr+d\equiv0\mod{p^\alpha}$. Therefore $\rho_g(p^\alpha)$ equals 1 if $p$ divides exactly one of $a$ and $c$ and 2 if it divides neither.

Having disposed of the first three cases, for the rest of the proof we may assume that $p\nmid ac$ and $p\mid\Delta$ (so that $\delta\ge1$). The congruence $g(r)\equiv0\mod{p^\alpha}$ is then equivalent to $(r + ba^{-1})(r + dc^{-1}) \equiv 0 \mod{p^\alpha}$. We may translate $r$ without affecting the number of roots modulo $p^\alpha$, so it is equivalent to look at the congruence $r(r + \Delta_1) \equiv 0 \mod{p^\alpha}$, where $\Delta_1 = dc^{-1} - ba^{-1} = \pm\Delta/ac$. To calculate $\rho_g(p^\alpha)$, we thus need to count the number of integers $1\le r\le p^\alpha$ such that $\ord_p(r) + \ord_p(r+\Delta_1) \ge \alpha$.

Note that $p^\delta\parallel\Delta_1$ as well (meaning that $p^\delta$ exactly divides its numerator while $p$ does not divide its denominator). The $p$-adic ultrametric inequality tells us that of the three numbers $\{ \ord_p(r), \ord_p(r+\Delta_1), \delta\}$, the two smallest are equal (or all three are equal). This implies that for any integer $r$, exactly one of the following situations holds:
\begin{itemize}
  \item $\ord_p(r) = \ord_p(r+\Delta_1) < \delta$; or
  \item $\min\{\ord_p(r),\ord_p(r+\Delta_1)\} = \delta$.
\end{itemize}

If $\alpha\le2\delta$, we apply Lemma~\ref{isolate the annoyance lemma} to see that $\ord_p(r) + \ord_p(r+\Delta_1) \ge \alpha$ exactly when $\lceil \frac\alpha2 \rceil \le \ord_p(r)$. The number of such $r$ between 1 and $p^\alpha$ is just the number of multiples of $p^{\lceil \alpha/2 \rceil}$ in that range, which is exactly $p^\alpha/p^{\lceil \alpha/2 \rceil} = p^{\lfloor \alpha/2 \rfloor}$. This settles the fourth case.

On the other hand, if $\alpha\le2\delta$, we apply Lemma~\ref{isolate the annoyance lemma} to see that $\min\{\ord_p(r),\ord_p(r+\Delta_1)\} = \delta$ and $\max\{\ord_p(r),\ord_p(r+\Delta_1)\} \ge \alpha-\delta>\delta$. For such integers $r$, either $r$ or $r+\Delta_1$ must be a multiple of $p^{\alpha-\delta}$ (and not both, since $\alpha-\delta>\delta$). The number of such $r$ between 1 and $p^\alpha$ is therefore exactly $2p^\alpha/p^ {\alpha-\delta} = 2p^\delta$. This settles the fifth and final case.
\end{proof}

\begin{lemma}
Let $f$ be a reducible, nonsquare quadratic polynomial with integer coefficients, and let $D$ be its discriminant. Then $\rho(m) \leq \sqrt{D} \cdot 2^{\omega(m)}$.
\label{D dependence lemma}
\end{lemma}

\begin{proof}
Note that $D\ne0$ since $g$ is not the square of a linear polynomial by assumption. As noted earlier in this section, we can write $f(t) = W(at+b)(ct+d)$ for some integers $W,a,b,c,d$ with $(a,b)=(c,d)=1$. Let $p$ be any prime, and let $\beta = \ord_p(\sqrt D)$, which is an integer since $D$ is a perfect square. Choose $\gamma$ and $\delta$ such that $p^\gamma \parallel W$ and $p^\delta \parallel (ad-bc)$, so that $\gamma+\delta=\beta$. If $\alpha\le\gamma$, then Lemma~\ref{evict content lemma} tells us that $\rho_f(p^\alpha) = p^\alpha \le p^\gamma \le p^\beta$. If $\alpha>\gamma$, then Lemmas~\ref{evict content lemma} and~\ref {primitive rho lemma} tell us that $\rho_f(p^\alpha) = p^\gamma \rho_f(p^{\alpha-\gamma}) \le p^\gamma\cdot 2p^\delta = 2p^\beta$. In either case we see that $\rho_f(p^\alpha) \le 2p^ {\ord_p(\sqrt D)}$. Then, since $\rho_f$ is multiplicative,
\[
\rho_f(m) = \prod_{p^\alpha\parallel m} \rho_f(p^\alpha) \le \prod_{p\mid m} \big( 2p^{\ord_p(\sqrt D)} \big) = \prod_{p\mid m} p^{\ord_p(\sqrt D)} \prod_{p\mid m} 2 \le \sqrt D \prod_{p\mid m} 2 = \sqrt D\cdot 2^{\omega(m)},
\]
as claimed.
\end{proof}

\subsection{Proof of Theorem~\ref{equidist thm} and related remarks}

We now describe our adaptation of Hooley's argument~\cite{hooley} to the case of reducible quadratic polynomials, which culminates in a proof of Theorem~\ref{equidist thm}. After the proof we make some additional comments on possible improvements to the theorem.

The only changes that need to be made to Hooley's argument~\cite{hooley} involve the differences in the function $\rho_f$ resulting from the fact that $f$ is now reducible. For one thing, the average size of $\rho_f$ is now logarithmic rather than constant; this does not ruin the argument, but it needs to be taken into account. On the other hand, we know $\rho_f$ much more exactly for reducible quadratic polynomials (as Lemma~\ref {primitive rho lemma} shows) than we do for general irreducible polynomials; consequently, we are able to be more explicit in some stages. We also wish to make all dependencies on $h$ and the polynomial $f$ explicit in our upper bounds, for the benefit of anyone wishing to employ Theorem~\ref{equidist thm} in the future. Even so, the number of changes is small relative to the several-page length of Hooley's argument. We have therefore, with apologies to the reader, decided not to include a self-contained proof but rather to indicate the necessary alterations to Hooley's proof. We will use some of the notation therein without definition when defining the notation is superfluous to the current account.

\begin{proof}[Outline of proof of Theorem~\ref{equidist thm}]
We begin by examining Hooley's auxiliary lemmata from~\cite{hooley}. His Lemma 1--Lemma 3 are valid for any polynomial, reducible or irreducible, as the proofs essentially depend only on the Chinese remainder theorem. We will not use his Lemma 4--Lemma 6: these deal with various estimates for the function $\rho_f$, whereas we will simply use the information worked out earlier in this section. Finally, his Lemma 7--Lemma 8 make no reference to the polynomial and thus remain valid in our setting. (We note that since we are dealing always with quadratic polynomials, Hooley's parameters $n$ and $N$ will equal 2 and 4, respectively, for us.) For the remainder of this proof, we use $x$ instead of $y$ so as to conform with the notation in~\cite{hooley}; this is not to be confused with the $x$ that appears in the rest of this paper.

Hooley's estimation proper of $R_f(h,x)$ begins with a certain decomposition~\cite[equation (1)]{hooley}, namely $R_f(h,x) = \Sigma_1 + \Sigma_2$, after which he observes that
\[
\Sigma_2 \ll \sum_{\substack{k\le x \\ k_1 > x^{1/3}}} \rho_f(k).
\]
We employ the upper bound in Lemma~\ref{D dependence lemma} to deduce that
\[
\Sigma_2 \ll \sqrt D \sum_{\substack{k\le x \\ k_1 > x^{1/3}}} 2^{\omega(k)},
\]
which differs from Hooley's estimate only in the presence of the $\sqrt D$ term. Thus no further changes are needed in the estimation of $\Sigma_2$: we obtain~\cite[equation (5)]{hooley}
\[
\Sigma_2 \ll \sqrt D \frac x{\log x},
\]
which will turn out to be smaller than our estimate of $\Sigma_1$.

Subsequently, Hooley derives the upper bound~\cite[equations (6) and (7)]{hooley}
\[
\Sigma_1 \ll \sum_{k_1 \le x^{1/3}} \sqrt{\Sigma_5\Sigma_6}.
\]
In the estimation of $\Sigma_5$, the only modification we need to make is to include a factor of $D$ in the estimate
\[
\rho_f^2(k_2) \ll D\cdot 2^{2\omega(k_2)} \ll D\cdot d_4(k_2)
\]
(the factor of $D$ resulting from applying Lemma~\ref{D dependence lemma} as in the estimation of $\Sigma_2$). Hooley's majorization of $\Sigma_6$ is exactly suitable for our purposes, except that we wish to keep explicit the dependence on $h$, so that we use his equation~(9) rather than his equation~(10). Our resulting version of~\cite[equation (11)]{hooley} is
\[
\Sigma_1 \ll \sqrt D \frac{x(\log\log x)^{5/2}}{\log x} \sum_{k_1\le x^{1/3}} \sqrt{ \frac{2^{\omega(k_1)}(h,k_1)}{k_1 \phi(k_1)} }.
\]
As Hooley does, we extend the range of summation on the right-hand side to all $\ell\le x$ (the notation hides the fact that the sum currently runs over those integers less than $x$ for which a certain divisor $k_1$ is at most $x^{1/3}$). The resulting sum is treated in Lemma~\ref{being anal with h lemma} below, whence we obtain
\[
\Sigma_1 \ll \sqrt D \frac{x(\log\log x)^{5/2}}{\log x} \cdot (\log x)^{\sqrt{2}} \prod_{p\mid h} \bigg( 1+\frac7{\sqrt p} \bigg).
\]
This establishes the theorem.
\end{proof}

Theorem~\ref{equidist thm} gives the estimate $R_f(h,y) \ll_{f,h} y (\log y)^{\sqrt{2}-1} (\log \log y)^{5/2}$ for any nonzero integer~$h$. On the other hand, the number of terms in the exponential sum is $\sum_{m\le y} \rho_f(m) \gg_f y\log y$ by equation~\eqref {rho_f asymptotic} below. In other words, the exponential sum exhibits nontrivial cancellation for all nonzero~$h$. By Weyl's criterion, this is precisely what is needed to show that the normalized roots of $f(t)\equiv0\mod m$ are equidistributed modulo~1.

Note that for a linear polynomial $f(t) = at+b$ with $(a,b)=1$, the normalized roots are not equidistributed modulo~1, since they cluster around the points $\frac sa$ with $(s,a)=1$. A straightforward calculation and an invocation of Ramanujan's sum shows that
\[
  \sum_{k \leq y} \sum_{\substack{a\nu+b \equiv 0 \mod{k} \\ 0 < \nu \leq k}} e^{2\pi i h \nu/k} = y \frac{\phi(a)}a \frac{\mu( a/(h,a) )}{\phi( a/(h,a) )} + O(h \log y),
\]
which is the same order of magnitude as the number of summands $y$, at least for some values of~$h$. On the other hand, the roots of the reducible quadratic $W(at+b)(ct+d)$ modulo~$m$ certainly include the roots of $at+b$ and $ct+d$, and so there will be a contribution to the exponential sum from these roots, whose order of magnitude is~$y$. It seems reasonable to conjecture that the other roots of the quadratic are distributed randomly. For example, with the polynomial $f(t) = t^2-1$, one would conjecture that
\[
  \sum_{k \leq y} \sum_{\substack{\nu^2-1 \equiv 0 \mod{k} \\ 0 < \nu \leq k}} e^{2\pi i h \nu/k} = 2y + O_h(y^{1/2+\ep}).
\]
In any event, we should not expect any estimate of the form $R_f(h,y) = o(y)$ for reducible quadratic polynomials~$f$; therefore the bound in Theorem~\ref{equidist thm} is not too far from what one could prove.

%\[
%  \sum_{(k,a)=1} e\bigg( \frac{hk}a \bigg) = \frac{\phi(a) \mu( a/(h,a) )}{\phi( a/(h,a) )}
%\]

We remark that the dependence on $D$ in the upper bound of Theorem~\ref{equidist thm} could be improved if necessary. There are two places in the proof of Theorem~\ref{equidist thm} where we use Lemma~\ref {D dependence lemma} to simply bound $\rho_f(n)$ by $\sqrt D\cdot 2^{\omega(n)}$; but for many values of $n$ the true size of $\rho_f(n)$ is much closer to $2^{\omega(n)}$. A more precise application of Lemmas~\ref {evict content lemma} and~\ref {primitive rho lemma} would replace the $\sqrt D$ in Theorem~\ref{equidist thm} by a smaller multiplicative function of $D$, one that was $\ll_\ep D^\ep$, for instance. Similarly, the dependence of the upper bound on $h$ could be slightly reduced by using the exact expression derived on the last line of equation~\eqref{better than 7} below.

\section{\dqs}
\label{main proof section}

With these two technical results in place, we can now proceed to the proof of Theorem~\ref{main thm}. It turns out that the analysis hinges on studying a very specific family of \dqs. A {\em\dr} \dq\ is one of the form
\begin{equation}
  \{ a, b, a+b+2r, 4r(a+r)(b+r) \},
\label{drdq def}
\end{equation}
where $a$, $b$, and $r$ are positive integers satisfying $a<b$ and $ab+1=r^2$ (so that $\{a,b\}$ is a Diophantine pair); it is easy to verify that any such quadruple is in fact Diophantine. Let
%$Q(x)$ denote the number of \drdqs\ contained in $[1,x]$.
\[
\text{$Q(x) = {}$the number of \drdqs\ contained in $[1,x]$.}
\]
Dujella proved that almost all \dqs\ are \dr; more precisely, he showed \cite[Section 4]{dujella} that the number of \dqs\ contained in $[1,x]$ is $Q(x) + O(x^{0.292}\log^2x)$. Therefore it suffices to find an asymptotic formula for $Q(x)$.

Define the set
\[
  R(m) = \{ \nu\colon 1 \leq \nu \leq m,\ \nu^2 \equiv 1 \mod{m} \}.
\]
Notice that in any \drdq, we have the congruence $r^2 = ab+1 \equiv1\mod b$. Moreover, the inequality $a<b$ forces $r\le b$ as well, so that $r$ must be an element of $R(b)$. Conversely, any pair $\{r,b\}$ with $r\in R(b)$ gives rise to a \drdq\ by taking $a=(r^2-1)/b$, except that $r=1$ gives rise to $a=0$ which must be excluded. We must therefore count all the pairs $\{r,b\}$, with $r\in R(b)\setminus\{1\}$, such that the largest element $4r(a+r)(b+r)$ of the corresponding \drdq\ is at most~$x$. In other words, if we define the function
$$
  L(a,r,b;x) = \begin{cases}
    1, & \text{if } 4r(a+r)(b+r) \le x, \\
    0, & \text{otherwise},
  \end{cases}
$$
then
\begin{equation}
  Q(x) = \sum_{b\in\N} \sum_{r \in R(b) \setminus \{ 1 \}} L \bigg( \frac{r^2-1}{b}, r, b; x \bigg).
\label{first Q form}
\end{equation}

The obstacle we must overcome is this: whether or not $4r(a+r)(b+r) \le x$ depends heavily on where in the interval $[1,b]$ the congruential root $r$ lies, when $b$ is in the most significant range (around $x^{1/3}$ in size). By replacing the summand in equation~\eqref{first Q form} by upper and lower bounds, Dujella was able \cite[Theorem 3]{dujella} to work out that the order of magnitude of $Q(x)$, and hence of the number of \dqs\ up to~$x$, is $x^{1/3}\log x$. We use our knowledge of the equidistribution of the roots $r\in R(b)$ to show that $Q(x)$ is asymptotically the same as an analogous sum where the numbers $r$ are chosen at random from $[1,b]$.

We use the notation $\rho(m) = \# R(m)$, which is a special case of the notation $\rho_g(m)$ from the last section with $g(t) = t^2-1$. Applying Lemma~\ref{primitive rho lemma} to this polynomial shows that
\begin{equation*}
\rho(p^\alpha) = \begin{cases}
2, & \text{if } p\ne 2 \text{ or } p^\alpha=4, \\
1, & \text{if } p^\alpha=2, \\
4, & \text{if } p=2 \text{ and } \alpha\ge3;
\end{cases}
\end{equation*}
in particular,
\begin{equation}
\rho(m) = \left. \begin{cases}
2^{\omega(m)}, & \text{if } 2\nmid m \text{ or } 2^2\parallel m, \\
2^{\omega(m)-1}, & \text{if } 2^1\parallel m, \\
2^{\omega(m)+1}, & \text{if } 2^3\mid m
\end{cases} \right\} \le 2^{\omega(m)+1}.
\label{special rho formula}
\end{equation}

\subsection{Truncating the infinite sum}

To begin with, we need to truncate the infinite sum in equation~\eqref{first Q form} in a manageable way. It turns out that we can accomplish this by sorting \drdqs\ by $a$ rather than~$b$.

\begin{lemma}
For any real numbers $A,x\ge2$, the number of \drdqs~\eqref{drdq def} satisfying $a\le A$ and $4r(a+r)(b+r) \le x$ is $\ll (Ax)^{1/4} \log A$.
\label{sort by a lemma}
\end{lemma}

\begin{proof}
As before, the congruence $r^2 = ab+1 \equiv1\mod a$ shows that $r$ must be congruent to some $\nu\in R(a)$, so that we can write $r=\nu + ak$ for some integer $k$. Furthermore, the inequality $a<b$ forces $a<r$ as well, so that $k\ge1$. Conversely, any such $r$ determines a \drdq\ by taking $b=(r^2-1)/a$. Note that the inequality $4r(a+r)(b+r) \le x$ implies that
\[
  x \ge 4 (\nu+ak)( a + \nu + ak ) \bigg( \frac{ (\nu+ak)^2 - 1 }{a} + \nu + ak \bigg) >  4 (ak) (ak) (ak^2) = 4a^3 k^4;
\]
in other words, it is necessary that $1\le k < (x/4a^3)^{1/4}$. We conclude that for every integer $a$ and every $\nu\in R(a)$, there are at most $(x/4a^3)^{1/4}$ corresponding \drdqs\ contained in $[1,x]$. An upper bound for the number of such quadruples with $a \le A$ is therefore
\begin{align*}
  \sum_{a \le A} \sum_{\nu \in R(a)} \bigg( \frac{x}{4a^3} \bigg)^{1/4}
     =  \sum_{a \le A} \bigg( \frac{x}{4a^3} \bigg)^{1/4} \rho(a) 
   \ll  x^{1/4} \sum_{a < A} \frac{\rho(a)}{a^{3/4}} 
   \ll  x^{1/4} A^{1/4} \log A,
\end{align*}
where the last step uses equation~\eqref {partial summation equation}.
\end{proof}

Let $\psi(x)\ge1$ be any function that tends to infinity as $x \to \infty$, but more slowly than $\log x$ (we will choose a specific function $\psi(x)$ later in equation~\eqref {psi choice}). The following proposition shows that we can truncate the sum in equation~\eqref{first Q form} at
\begin{equation}
  B = B(x) = \lfloor x^{1/3}\psi(x) \rfloor.
\label{B def}
\end{equation}

\begin{proposition}
For any real number $x\ge3$,
\[
  Q(x) = \sum_{b\le B} \sum_{r \in R(b)} L \bigg( \frac{r^2-1}{b}, r, b; x \bigg) + O \bigg( x^{1/3} \Big( \psi(x) + \frac{\log x}{\psi(x)^{1/2}} \Big) \bigg),
\]
where $B$ is defined in equation~\eqref{B def}.
\label{truncated prop}
\end{proposition}

\begin{proof}
We begin by bounding the number of \drdqs\ with $b> B$ and $4r(a+r)(b+r) \le x$. Since $r^2 = ab+1$, these inequalities force $x \ge 4r(a+r)(b+r) > 4r^2b > 4(ab)b > 4aB^2$; in other words, every such \dq\ satisfies $a < x/4B^2 \le x^{1/3}/\psi(x)^2$. By Lemma~\ref{sort by a lemma}, the number of such \dqs\ is
\[
\ll \bigg( \frac{x^{1/3}}{\psi(x)^2}\cdot x \bigg)^{1/4} \log \bigg( \frac{x^{1/3}}{\psi(x)^2} \bigg) \le \frac{x^{1/3}\log x}{\psi(x)^{1/2}}.
\]
Therefore equation~\eqref{first Q form} becomes
\begin{align*}
  Q(x) &= \sum_{b\le B} \sum_{r \in R(b) \setminus \{ 1 \}} L \bigg( \frac{r^2-1}{b}, r, b; x \bigg) + O\bigg( \frac{x^{1/3}\log x}{\psi(x)^{1/2}} \bigg) \\
  &= \sum_{b\le B} \sum_{r \in R(b)} L \bigg( \frac{r^2-1}{b}, r, b; x \bigg) + O \bigg( B + \frac{x^{1/3}\log x}{\psi(x)^{1/2}} \bigg),
\end{align*}
which is equivalent to the statement of the proposition.
\end{proof}

\subsection{Invoking equidistribution}

\begin{lemma}
Let $a,b,r$ be positive integers with $a<b$ and $ab+1=r^2$. Then for any real number $x\ge3$, the inequality $4r(a+r)(b+r)\le x$ is equivalent to $\frac rb \le \lambda(b,x)$, where
\begin{equation}
  \lambda(b,x) = \begin{cases}
    1, &\text{if } b \le (x/16)^{1/3}, \\
    \displaystyle \sqrt{\frac{x^{1/2}}{2b^{3/2}} + \frac14}-\frac12 + O\bigg( \frac b{x^{3/4}} \bigg), &\text{if } b > (x/16)^{1/3}.
  \end{cases}
\label{lambda def}
\end{equation}
\label{length function lemma}
\end{lemma}

\begin{proof}
The inequality $4r(a+r)(b+r) \le x$ is the same as
$$
  \frac{r}{b} \bigg( \frac{r^2}{b^2} + \frac{r}{b} - \frac{1}{b^2} \bigg)
  \bigg( 1 + \frac{r}{b} \bigg) \le \frac{x}{4b^3}.
$$
If we set $s = \frac{r}{b}$, which is the normalized root used in the statement of our equidistribution result, then this becomes
\begin{equation}
  s^2 (1+s)^2 \le \frac{x}{4b^3} + \frac{s(1+s)}{b^2} = \frac{x}{4b^3} \bigg( 1 + O\bigg( \frac b{x} \bigg) \bigg).
  \label{changed to s}
\end{equation}
Note that $s\le 1$ and so the left-hand side is at most~4. If it happens that $b \le (x/16)^{1/3}$, then the first term on the right-hand side is at least~4, so that the inequality always holds; setting $\lambda(b,x)=1$ is therefore valid in this region. Otherwise, for any positive real number $y$, we see by completing the square that $s^2(1+s)^2 \le y$
if and only if
$
  s \le (\sqrt y + 1/4)^{1/2} - 1/2.
$
Using this equivalence, the inequality~\eqref{changed to s} becomes
\begin{align*}
  s &\le \bigg( \sqrt{\frac{x}{4b^3} + \frac{s(1+s)}{b^2}} + \frac{1}{4} \bigg)^{1/2} -\frac{1}{2} \\
  &= \bigg( \sqrt{\frac{x}{4b^3} \bigg( 1 + O\bigg( \frac b{x} \bigg) \bigg)} + \frac{1}{4} \bigg)^{1/2} -\frac{1}{2} \\
  &= \bigg( \sqrt{\frac{x}{4b^3}} + \frac{1}{4} \bigg)^{1/2} \bigg( 1 + O\bigg( \frac b{x} \bigg) \bigg) -\frac{1}{2} \\
  &= \bigg( \sqrt{\frac{x}{4b^3}} + \frac{1}{4} \bigg)^{1/2} -\frac{1}{2} + O\bigg( \frac b{x^{3/4}} \bigg),
\end{align*}
since $b\ge1$. This establishes the lemma in the case where $b > (x/16)^{1/3}$.
\end{proof}

\begin{proposition}
For any real number $x\ge3$,
\begin{equation*}
 Q(x) = \sum_{b\le B} \rho(b) \lambda(b,x) + O \bigg( x^{1/3}\psi(x) (\log x)^{\sqrt{2}/2} (\log \log x)^{5/4} + \frac{x^{1/3} \log x}{\psi(x)^{1/2}} \bigg),
\end{equation*}
where $B$ is defined in equation~\eqref{B def}.
\label{probabilistic model prop}
\end{proposition}

\begin{proof}
Consider the concatenation of $B$ finite sequences, the $b$th of which consists of the roots of $t^2\equiv1\mod b$ normalized by dividing by $b$; in other words, consider the sequence
$R(1) \cup \frac12R(2) \cup \frac13R(3) \cup \cdots \cup \frac1BR(B)$. This sequence has $S(B)$ elements, where we define
\begin{equation*}
  S(y) = \sum_{b \le y} \rho(b).
\end{equation*}
We will apply the moving-target Erd\H os-Tur\'an inequality, Theorem~\ref {moving target thm}, to this sequence; our target intervals will be $[0,\lambda(b,x)]$ for each element $r/b$ of $\frac1bR(b)$. With this setup, we have
\[
Z_{S(B)} = \sum_{b\le B} \#\{r\in R(b)\colon 0\le \tfrac rb\le \lambda(b,x) \}= \sum_{b\le B} \sum_{r \in R(b)} L \bigg( \frac{r^2-1}{b}, r, b; x \bigg)
\]
by Lemma~\ref {length function lemma}, whence
$$
  D_{S(B)} = \sum_{b\le B} \sum_{r \in R(b)} L \bigg( \frac{r^2-1}{b}, r, b; x \bigg)
    - \sum_{b\le B} \rho(b) (\lambda(b,x) - 0).
$$
In other words,
\[
  Q(x) = \sum_{b\le B} \rho(b) \lambda(b,x) + O \bigg( D_{S(B)} + x^{1/3} \Big( \psi(x) + \frac{\log x}{\psi(x)^{1/2}} \Big) \bigg)
\]
by Proposition~\ref{truncated prop} (recalling that $B$ is defined in equation~\eqref{B def}).

Our intervals $[0,\lambda(b,x)]$ have the property that the lower and upper endpoints each form monotone sequences; the variation in the lower endpoint is~0, while the variation in the upper endpoint is $1-\lambda(B,x) \le 1$. By Corollary~\ref{moving target cor}, we therefore have
\[
  D_{S(B)} \ll \frac{S(B)}H + \big( 1+ 0 + 1 \big) \sum_{h=1}^H M_ {S(B)}(h) \ll \frac {S(B)}H + \sum_{h=1}^H M_ {S(B)}(h)
\]
for any positive integer $H$; we will choose $H$ later in~\eqref{H choice}. By Lemma~\ref{rho asymptotic lemma}, the first term satisfies
\[
\frac {S(B)}H \ll \frac{B\log B}H \ll \frac{x^{1/3}\psi(x)\log x}H.
\]
By Theorem~\ref{equidist thm} applied with $y\le B$ and with $f(t)=t^2-1$, so that $D=2$,
\[
  M_ {S(B)}(h) \ll \prod_{p\mid h} \bigg( 1+\frac7{\sqrt p} \bigg) \cdot B (\log B)^{\sqrt{2}-1} (\log \log B)^{5/2} + \max_{1\le b\le B} \rho(b),
\]
the final term arising because we have to consider what happens if we chop off the exponential sum in the middle of one of the finite sequences $\frac1b R(b)$; here even the crude bound $\rho(b)\le B$ suffices. Consequently, by Lemma~\ref {h payoff lemma} we have
\[
  \sum_{h\le H} M_ {S(B)}(h) \ll H B (\log B)^{\sqrt{2}-1} (\log \log B)^{5/2}.
\]
We conclude that
\begin{align}
  Q(x) &= \sum_{b\le B} \rho(b) \lambda(b,x) + O \bigg( \frac {x^{1/3}\psi(x)\log x} H + HB (\log B)^{\sqrt{2}-1} (\log \log B)^{5/2} + x^{1/3} \Big( \psi(x) + \frac{\log x}{\psi(x)^{1/2}} \Big) \bigg) \notag \\
  &= \sum_{b\le B} \rho(b) \lambda(b,x) + O \bigg( \frac {x^{1/3}\psi(x)\log x} H + Hx^{1/3}\psi(x) (\log x)^{\sqrt{2}-1} (\log \log x)^{5/2} + \frac{x^{1/3} \log x}{\psi(x)^{1/2}} \bigg).
\end{align}

To optimize this error term, we choose $H$ so that the first two terms are the same size.  This choice turns out to be
\begin{equation}
  H = \big\lceil (\log x)^{1-\sqrt{2}/2} (\log \log x)^{-5/4} \big\rceil,
\label{H choice}
\end{equation}
and with this choice, we get
\begin{equation*}
 Q(x) = \sum_{b\le B} \rho(b) \lambda(b,x) + O \bigg( x^{1/3}\psi(x) (\log x)^{\sqrt{2}/2} (\log \log x)^{5/4} + \frac{x^{1/3} \log x}{\psi(x)^{1/2}} \bigg).
\end{equation*}
\end{proof}

\subsection{Calculating the weighted sum}

We evaluate the sum
$$
  \sum_{b \le B} \rho(b) \lambda(b,x)
$$
by partial summation, recalling that $B=\lfloor x^{1/3}\psi(x) \rfloor$ where $\psi(x)$ tends to infinity but more slowly than $\log x$; we also recall the notation $S(y) = \sum_{b \le y} \rho(b)$. After working hard to evaluate the eventual leading constant in closed form, we finally finish the proof of Theorem~\ref{main thm} at the end of this section.

\begin{proposition}
For any real number $x\ge3$,
\[
  \sum_{b\le B} \rho(b) \lambda(b,x) = \frac{3x^{1/2}}4 \int_{(x/16)^{1/3}}^\infty \bigg( 1+\frac{2x^{1/2}}{t^{3/2}} \bigg)^{-1/2} \frac{S(t)}{t^{5/2}} \, dt + O\bigg( \frac{x^{1/3}\log x}{\psi(x)^{1/2}} \bigg).
\]
\label{partial summation prop}
\end{proposition}

\begin{proof}
For notational convenience in this proof, we set $B_1 = (x/16)^{1/3}$. Using the definition~\eqref {lambda def} of $\lambda(b,x)$, we have
\begin{align}
  \sum_{b\le B} \rho(b) \lambda(b,x) &= \sum_{b\le B_1} \rho(b) + \sum_{B_1<b\le B} \rho(b) \bigg( \sqrt{\frac{x^{1/2}}{2b^{3/2}} + \frac14}-\frac12 + O\bigg( \frac b{x^{3/4}} \bigg) \bigg) \notag \\
  &= S(B_1) + \sum_{B_1<b\le B} \rho(b) \sqrt{\frac{x^{1/2}}{2b^{3/2}} + \frac14} - \tfrac12 ( S(B)-S(B_1) ) + O\bigg( \frac B{x^{3/4}} S(B) \bigg) \notag \\
  &= \sum_{B_1<b\le B} \rho(b) \sqrt{\frac{x^{1/2}}{2b^{3/2}} + \frac14} + \tfrac32S(B_1) -\tfrac12 S(B) + O(1).
\label{pre partial summation}
\end{align}
The remaining sum can be written as a Riemann-Stieltjes integral, to which integration by parts can be applied:
\begin{align}
  \sum_{B_1<b\le B} \rho(b) \sqrt{\frac{x^{1/2}}{2b^{3/2}} + \frac14} &= \int_{B_1}^B \bigg( \frac{x^{1/2}}{2t^{3/2}} + \frac14 \bigg)^{1/2} \, dS(t) \notag \\
  &= S(B) \bigg( \frac{x^{1/2}}{2B^{3/2}} + \frac14 \bigg)^{1/2} - S(B_1) \bigg( \frac{x^{1/2}}{2B_1^{3/2}} + \frac14 \bigg)^{1/2} \label{integrate by parts} \\
  &\qquad{}- \int_{B_1}^B S(t) \frac d{dt} \bigg( \frac{x^{1/2}}{2t^{3/2}} + \frac14 \bigg)^{1/2} \, dt. \notag
\end{align}
We note that
\[
  \bigg( \frac{x^{1/2}}{2B^{3/2}} + \frac14 \bigg)^{1/2} = \bigg( \frac1{2\psi(x)^{3/2}} + \frac14 \bigg)^{1/2} = \frac12+O\bigg( \frac1{\psi(x)^{3/2}} \bigg),
\]
while the similar term with $B$ replaced by $B_1$ evaluates to exactly $\frac32$; we also note that
\[
  \frac d{dt} \bigg( \frac{x^{1/2}}{2t^{3/2}} + \frac14 \bigg)^{1/2} = - \frac{3x^{1/2}}{8t^{5/2}} \bigg( \frac{x^{1/2}}{2t^{3/2}} + \frac14 \bigg)^{-1/2}.
\]
Therefore equation~\eqref{integrate by parts} becomes
\[
  \sum_{B_1<b\le B} \rho(b) \sqrt{\frac{x^{1/2}}{2b^{3/2}} + \frac14} = \big( \tfrac12+O( \psi(x)^{-3/2}) \big) S(B) - \tfrac32 S(B_1) + \tfrac38x^{1/2} \int_{B_1}^B \frac{S(t)}{t^{5/2}} \bigg( \frac{x^{1/2}}{2t^{3/2}} + \frac14 \bigg)^{-1/2} \, dt,
\]
which results in a lot of cancellation in equation~\eqref {pre partial summation}:
\begin{align*}
  \sum_{b\le B} \rho(b) \lambda(b,x) &= \tfrac38x^{1/2} \int_{B_1}^B \frac{S(t)}{t^{5/2}} \bigg( \frac{x^{1/2}}{2t^{3/2}} + \frac14 \bigg)^{-1/2} \, dt + O\bigg( \frac{S(B)}{\psi(x)^{3/2}} \bigg) \\
  &= \tfrac34x^{1/2} \int_{B_1}^B \frac{S(t)}{t^{5/2}} \bigg( \frac{2x^{1/2}}{t^{3/2}} + 1 \bigg)^{-1/2} \, dt + O\bigg( \frac{B\log B}{\psi(x)^{3/2}} \bigg)
\end{align*}
by Lemma~\ref {rho asymptotic lemma}. Finally, we extend the integral to infinity, noting that
\[
  \int_{B}^\infty \frac{S(t)}{t^{5/2}} \bigg( \frac{2x^{1/2}}{t^{3/2}} + 1 \bigg)^{-1/2} \, dt \ll \int_{B}^\infty \frac{t\log t}{t^{5/2}} = 2B^{-1/2}\log B + 4B^{-1/2} \ll \frac{\log B}{B^{1/2}},
\]
so that
\begin{equation*}
  \sum_{b\le B} \rho(b) \lambda(b,x) = \tfrac34x^{1/2} \int_{B_1}^\infty \frac{S(t)}{t^{5/2}} \bigg( \frac{2x^{1/2}}{t^{3/2}} + 1 \bigg)^{-1/2} \, dt + O\bigg( \frac{B\log B}{\psi(x)^{3/2}} + \frac{x^{1/2}\log B}{B^{1/2}} \bigg).
\end{equation*}
Since $B=\lfloor x^{1/3}\psi(x) \rfloor$, both error terms are $\ll x^{1/3}(\log x)/\psi(x)^{1/2}$, which establishes the proposition.
\end{proof}

\begin{lemma}
For any real number $x\ge3$,
\[
  \frac{3x^{1/2}}4 \int_{(x/16)^{1/3}}^\infty \bigg( 1+\frac{2x^{1/2}}{t^{3/2}} \bigg)^{-1/2} \frac{S(t)}{t^{5/2}} \, dt = \frac {2^{2/3}}{\pi^2} x^{1/3}\log x \int_0^1 (1+8u)^{-1/2} u^{-2/3}\, du + O(x^{1/3}).
\]
\label{integral with S(t) lemma}
\end{lemma}

\begin{proof}
First we use the asymptotic formula for $S(t)$ in Lemma~\ref {rho asymptotic lemma}:
\begin{multline*}
  \int_{(x/16)^{1/3}}^\infty \bigg( 1+\frac{2x^{1/2}}{t^{3/2}} \bigg)^{-1/2} \frac{S(t)}{t^{5/2}} \, dt \\
  {}= \frac6{\pi^2} \int_{(x/16)^{1/3}}^\infty \bigg( 1+\frac{2x^{1/2}}{t^{3/2}} \bigg)^{-1/2} \frac{\log t}{t^{3/2}} \, dt + O\bigg( \int_{(x/16)^{1/3}}^\infty \bigg( 1+\frac{2x^{1/2}}{t^{3/2}} \bigg)^{-1/2} \frac{1}{t^{3/2}} \, dt \bigg).
\end{multline*}
Since $(1+2x^{1/2}/t^{3/2})^{-1/2} < 1$, the last integral is bounded above by $2((x/16)^{1/3})^{-1/2}$. Therefore
\[
  \frac{3x^{1/2}}4 \int_{(x/16)^{1/3}}^\infty \bigg( 1+\frac{2x^{1/2}}{t^{3/2}} \bigg)^{-1/2} \frac{S(t)}{t^{5/2}} \, dt = \frac {9x^{1/2}}{2\pi^2} \int_{(x/16)^{1/3}}^\infty \bigg( 1+\frac{2x^{1/2}}{t^{3/2}} \bigg)^{-1/2} \frac{\log t}{t^{3/2}} \, dt + O(x^{1/3}).
\]
We now make the change of variables $u = x^{1/2}/4t^{3/2}$ in the remaining integral, which yields
\begin{align*}
  \frac {9x^{1/2}}{2\pi^2} \int_{(x/16)^{1/3}}^\infty & \bigg( 1+\frac{2x^{1/2}}{t^{3/2}} \bigg)^{-1/2} t^{-3/2}\log t \, dt \\
  &= \frac {9x^{1/2}}{2\pi^2} \int_1^0 (1+8u)^{-1/2} \frac{4u}{x^{1/2}} \log\frac{x^{1/3}}{2^{4/3}u^{2/3}} \bigg( {-}\frac{x^{1/3}}{3\cdot 2^{1/3}u^{5/3}} \, du \bigg) \\
  &= \frac {2^{2/3}x^{1/3}}{\pi^2} \bigg( (\log x) \int_0^1 (1+8u)^{-1/2}u^{-2/3} \, du + 2 \int_0^1 (1+8u)^{-1/2} \frac{\log4u}{u^{2/3}} \, du \bigg)
\end{align*}
This establishes the lemma, on noting that the last integral converges to some finite constant.
\end{proof}

\begin{lemma}
We have
\[
  \frac {2^{2/3}}{\pi^2} \int_0^1 (1+8u)^{-1/2} u^{-2/3}\, du = \frac{2^{4/3}}{3\Gamma(\tfrac23)^3} \approx 0.338285.
\]
\label{final constant lemma}
\end{lemma}

\begin{proof}
This lemma turns out to be an exercise in mining known results about special functions. We begin with Euler's integral representation~\cite[equation (15.3.1)]{AS} for the hypergeometric function $\F$, valid for all complex numbers $z$ other than real numbers greater than or equal to~1, as long as $\Re c>\Re b>0$:
\[
  \F(a,b;c;z) = \frac{\Gamma(c)}{\Gamma(b)\Gamma(c-b)} \int_0^1 t^{b-1}(1-t)^{c-b-1}(1-tz)^{-a}\, dt
\]
Choosing $a=\frac12$, $b=\frac13$, $c=\frac43$, and $z=-8$, and recalling that $\Gamma(y+1)=y\Gamma(y)$, we see that
\[
  \frac{2^{2/3}}{\pi^2} \int_0^1 (1+8u)^{-1/2} u^{-2/3}\, du = \frac{2^{2/3}}{\pi^2} \frac{\Gamma(\tfrac13) \Gamma(1)}{\Gamma(\tfrac43)} \F\big(\tfrac12, \tfrac13 ; \tfrac43 ; -8\big) = \frac{3\cdot 2^{2/3}}{\pi^2} \F\big(\tfrac12, \tfrac13 ; \tfrac43 ; -8\big).
\]
Next, we apply the quadratic transformation \cite[equation (15.3.22)]{AS}
\[
  \F\big(a,b;a+b+\tfrac12;z\big) = \F\big(2a,2b;a+b+\tfrac12;\tfrac12-\tfrac12\sqrt{1-z}\big)
\]
to obtain
\[
  \frac{3\cdot 2^{2/3}}{\pi^2} \F\big(\tfrac12, \tfrac13 ; \tfrac43 ; -8\big) = \frac{3\cdot 2^{2/3}}{\pi^2} \F\big(1, \tfrac23 ; \tfrac43 ; -1\big).
\]
At this point, we can actually evaluate this special value by the formula~\cite[equation (15.1.21)]{AS}
\[
  \F(a,b;a-b+1;-1) = 2^{-a} \pi^{1/2} \frac{\Gamma(1+a-b)}{\Gamma(1+\tfrac a2-b)\Gamma(\tfrac 12+ \tfrac a2)},
\]
which gives
\[
  \frac{3\cdot 2^{2/3}}{\pi^2} \F\big(1, \tfrac23 ; \tfrac43 ; -1\big) = \frac{3}{2^{1/3}\pi^{3/2}} \frac{\Gamma(\tfrac 43)}{\Gamma(\tfrac 56) \Gamma(1)} = \frac{1}{2^{1/3}\pi^{3/2}} \frac{\Gamma(\tfrac 13)}{\Gamma(\tfrac 56)}.
\]

We now invoke the duplication formula for the Gamma function~\cite[equation (6.1.18)]{AS}, namely $\Gamma(2z) = (2\pi)^{-1/2} 2^{2z-1/2} \Gamma(z) \Gamma\big( z+\tfrac12 \big)$. At $z= \tfrac 13$ this yields $\Gamma(\tfrac56) ={2^{1/3} \pi^{1/2} \Gamma(\tfrac23)/\Gamma(\tfrac13)}$, so that
\[
  \frac{1}{2^{1/3}\pi^{3/2}} \frac{\Gamma(\tfrac 13)}{\Gamma(\tfrac 56) \Gamma(1)} = \frac{1}{2^{2/3}\pi^2} \frac{\Gamma(\tfrac 13)^2}{\Gamma(\tfrac23)}
\]
Finally, the reflection formula for the Gamma function~\cite[equation (6.1.17)]{AS} is $\Gamma(z) \Gamma(1-z) = \pi\csc\pi z$; again choosing $z= \tfrac 13$, we obtain $\Gamma(\tfrac13) =\pi \csc\tfrac\pi3/{\Gamma(\tfrac23)} ={2\pi/3^{1/2} \Gamma(\tfrac23)}$, so that
\[
  \frac{1}{2^{2/3}\pi^2} \frac{\Gamma(\tfrac 13)^2}{\Gamma(\tfrac23)} = \frac{2^{4/3}}{3\Gamma(\tfrac23)^3}
\]
This chain of equalities establishes the lemma.
\end{proof}

\begin{proof}[Proof of Theorem~\ref{main thm}]
By Proposition~\ref {probabilistic model prop},
\begin{equation*}
 Q(x) = \sum_{b\le B} \rho(b) \lambda(b,x) + O \bigg( x^{1/3}\psi(x) (\log x)^{\sqrt{2}/2} (\log \log x)^{5/4} + \frac{x^{1/3} \log x}{\psi(x)^{1/2}} \bigg).
\end{equation*}
By Proposition~\ref {partial summation prop},
\begin{equation*}
 Q(x) = \frac{3x^{1/2}}4 \int_{(x/16)^{1/3}}^\infty \bigg( 1+\frac{2x^{1/2}}{t^{3/2}} \bigg)^{-1/2} \frac{S(t)}{t^{5/2}} \, dt + O \bigg( x^{1/3}\psi(x) (\log x)^{\sqrt{2}/2} (\log \log x)^{5/4} + \frac{x^{1/3} \log x}{\psi(x)^{1/2}} \bigg).
\end{equation*}
At this point we choose
\begin{equation}
\psi(x) = (\log x)^{(2-\sqrt2)/3}(\log\log x)^{-5/6}
\label{psi choice}
\end{equation}
to optimize the error term, yielding
\begin{equation*}
 Q(x) = \frac{3x^{1/2}}4 \int_{(x/16)^{1/3}}^\infty \bigg( 1+\frac{2x^{1/2}}{t^{3/2}} \bigg)^{-1/2} \frac{S(t)}{t^{5/2}} \, dt + O \big( x^{1/3} (\log x)^{2/3+\sqrt{2}/6} (\log \log x)^{5/12} \big).
\end{equation*}
By Lemma~\ref {integral with S(t) lemma},
\begin{equation*}
 Q(x) = \frac {2^{2/3}}{\pi^2} x^{1/3}\log x \int_0^1 (1+8u)^{-1/2} u^{-2/3}\, du + O \big( x^{1/3} (\log x)^{2/3+\sqrt{2}/6} (\log \log x)^{5/12} \big).
\end{equation*}
Finally, by Lemma~\ref {final constant lemma},
\begin{equation*}
 Q(x) = \frac{2^{4/3}}{3\Gamma(\tfrac23)^3} x^{1/3}\log x + O \big( x^{1/3} (\log x)^{2/3+\sqrt{2}/6} (\log \log x)^{5/12} \big).
\end{equation*}
\end{proof}

\section{Sums of multiplicative functions}
\label{mult fn section}

In this section we gather together the facts about sums of multiplicative functions that we used in Sections~\ref{equidist section} and~\ref{main proof section}. All of the specific results we need are special cases of the following asymptotic formula, several variants of which have appeared in the literature.

\begin{proposition}
Let $g(n)$ be a nonnegative multiplicative function. Suppose that there exist real numbers $U$ and $\kappa$ such that $g(p^\alpha)\le U$ for all prime powers $p^\alpha$ and
\[
\sum_{p\le w} \frac{g(p)\log p}p = \kappa\log w + O_g(1)
\]
for all $w\ge2$. Then the asymptotic formula
\[
\sum_{n\le y} \frac{g(n)}n = c(g)\log^\kappa y + O_g(\log^{\kappa-1}y)
\]
holds for all $y\ge2$, where $c(g)$ is the convergent product
\[
c(g) = \frac1{\Gamma(\kappa+1)} \prod_p \bigg(1-\frac1p\bigg)^\kappa \bigg( 1 + \frac{g(p)}p + \frac{g(p^2)}{p^2} + \cdots \bigg).
\]
\label{prop A3a}
\end{proposition}

\begin{proof}
This is exactly \cite[Proposition A.3(a)]{behemoth}, except that we have replaced the hypothesis ``$g(n)\ll n^\alpha$ for some constant $\alpha<1/2$'' with ``there exists $U$ such that $g(p^\alpha)\le U$ for all prime powers $p^\alpha$''. This is a strictly stronger hypothesis, however: any nonnegative multiplicative function satisfying $g(p^\alpha)\le U$ automatically satisfies $g(n)\le U^{\omega(n)} \ll_{U,\ep} n^\ep$ for every $\ep>0$.
\end{proof}

In the following lemmas we use the standard notations $\phi(n)$ for the Euler phi-function, $\mu(n)$ for the M\"obius mu-function, and $\omega(n)$ for the number of distinct prime divisors of $n$. Note that $\mu^2$ is the characteristic function of the squarefree integers.

\begin{lemma}
For all $y\ge2$, we have $\displaystyle \sum_{n \leq y} \sqrt{ \frac{2^{\omega(n)}}{n \phi(n)} } \ll (\log y)^{\sqrt2}$.
\label{sqrt2 lemma}
\end{lemma}

\begin{proof}
Define the multiplicative function
$$
  g(n) = \sqrt{\frac{n 2^{\omega(n)}}{\phi(n)}} = \prod_{p\mid n} \sqrt{\frac{2p}{p-1}},
$$
so that the sum we need to estimate is $\sum_{n \leq y}{g(n)}/{n}$. We note that $g(p^\alpha) = \sqrt{2p/(p-1)} \le 2$ for all prime powers $p^\alpha$, and we evaluate
\begin{align*}
  \sum_{p \leq y} \frac{g(p) \log p}{p} & =
    \sum_{p \leq y} \sqrt{\frac{2p}{p-1}} \frac{\log p}{p} \\
  & =  \sqrt{2} \sum_{p \leq y} \frac{\log p}{p \sqrt{1-1/p}} \\
  & =  \sqrt{2} \sum_{p \leq y} \frac{\log p}{p} \bigg( 1 + O
    \bigg( \frac{1}{p} \bigg) \bigg) \\
  & =  \sqrt{2} \sum_{p \leq y} \frac{\log p}{p} +
    O \bigg( \sum_{p \leq y} \frac{\log p}{p^2} \bigg) \\
  & =  \sqrt{2}\log y + O(1).
\end{align*}
Proposition~\ref{prop A3a} therefore applies with $\kappa=\sqrt2$, yielding
$$
  \sum_{n \leq y} \sqrt{\frac{2^{\omega(n)}} {n \phi(n)}} = \sum_{n \leq y} \frac{g(n)}{n} = c(g) (\log y)^{\sqrt{2}} + O_g((\log y)^{\sqrt{2}-1}) \ll (\log y)^{\sqrt{2}}
$$
as claimed (the $\ll$-constant is absolute since the function $g$ is fixed).
\end{proof}

\begin{lemma}
For all $y\ge2$ and for any nonzero integer $h$,
\[
\sum_{n \leq y} \sqrt{ \frac{2^{\omega(n)}(h,n)}{n \phi(n)} } \ll (\log y)^{\sqrt{2}} \prod_{p\mid h} \bigg( 1+\frac7{\sqrt p} \bigg)
\]
\label{being anal with h lemma}
\end{lemma}

\begin{proof}
We sort the integers $n\le y$ according to their greatest common divisor $d$ with $h$:
\[
\sum_{n \leq y} \sqrt{ \frac{2^{\omega(n)}(h,n)}{n \phi(n)} } = \sum_{d\mid h} \sqrt d \sum_{\substack{n\le y \\ (n,h)=d}} \sqrt{ \frac{2^{\omega(n)}}{n \phi(n)} } = \sum_{d\mid h} \sqrt d \sum_{\substack{m\le y/d \\ (m,h)=1}} \sqrt{ \frac{2^{\omega(dm)}}{dm \phi(dm)} }.
\]
The condition $(m,h)=1$ implies that $(m,d)=1$, and so
\[
\sum_{n \leq y} \sqrt{ \frac{2^{\omega(n)}(h,n)}{n \phi(n)} } = \sum_{d\mid h} \sqrt{\frac{2^{\omega(d)}}{\phi(d)}} \sum_{\substack{m\le y/d \\ (m,h)=1}} \sqrt{ \frac{2^{\omega(m)}}{m \phi(m)} } \le \sum_{d\mid h} \sqrt{\frac{2^{\omega(d)}}{\phi(d)}} \sum_{m\le y} \sqrt{ \frac{2^{\omega(m)}}{m \phi(m)} }.
\]
The inner sum is $\ll(\log y)^{\sqrt2}$ by Lemma~\ref {sqrt2 lemma}, and so it remains to bound the sum over $d$. But this is a multiplicative function of $h$, and so
\begin{align}
  \sum_{d\mid h} \sqrt{\frac{2^{\omega(d)}}{\phi(d)}} &= \prod_{p^\alpha\parallel h} \bigg( 1 + \sqrt{\frac{2^{\omega(p)}}{\phi(p)}} + \cdots + \sqrt{\frac{2^{\omega(p^\alpha)}}{\phi(p^\alpha)}} \bigg) \notag \\
  &= \prod_{p^\alpha\parallel h} \bigg( 1 + \sqrt{\frac2{p-1}} + \cdots + \sqrt{\frac2{p^{r-1}(p-1)}} \bigg) \notag \\
  &= \prod_{p^\alpha\parallel h} \bigg( 1 + \sqrt{\frac2{p-1}} \sum_{j=0}^{r-1} \frac1{p^{j/2}} \bigg) \notag \\
  &\le \prod_{p\mid h} \bigg( 1 + \sqrt{\frac2{p-1}} \big( 1-p^{-1/2} \big)^{-1} \bigg).
\label{better than 7}
\end{align}
The lemma follows upon verifying that $\sqrt{2/(p-1)}( 1-p^{-1/2})^{-1} \le 7/\sqrt p$ for all $p\ge2$.
\end{proof}

For certain multiplicative functions $g$, we can actually find an asymptotic formula not just for $\sum_{m\le y} g(m)/m$ but also for $\sum_{m\le y} g(m)$. In our next lemma we record only the upper bound, even though we can derive an asymptotic formula; in the lemma after that, the asymptotic formula is important enough to retain.

\begin{lemma}
For all $y\ge2$, we have $\displaystyle \sum_{m \le y} \prod_{p\mid m} \bigg( 1+\frac7{\sqrt p} \bigg) \ll y$.
\label{h payoff lemma}
\end{lemma}

\begin{proof}
Define the multiplicative function $g(m) = 7^{\omega(m)}\mu^2(m)/\sqrt m$, where $\mu$ is the M\"obius mu-function. One can check that
\[
\prod_{p\mid m} \bigg( 1+\frac7{\sqrt p} \bigg) = \sum_{d\mid m} g(d)
\]
(because both sides are multiplicative functions of $m$, it suffices to check the equality on prime powers). We have
\begin{equation}
  \sum_{m\le y} \prod_{p\mid m} \bigg( 1+\frac7{\sqrt p} \bigg) = \sum_{m\le y} \sum_{d\mid m} g(d) 
  = \sum_{d\le y} g(d) \sum_{\substack{m\le y \\ d\mid m}} 1 
  = y\sum_{d\le y} \frac{g(d)}d + O\bigg( \sum_{d\le y} g(d) \bigg).
  \label{divisor trick 1}
\end{equation}
Since $g(m) \ll_\ep m^{-1/2+\ep}$ for any $\ep>0$, the error term is $O_\ep(y^{1/2+\ep})$. We note that $g(p^\alpha) < 7$ for all prime powers $p^\alpha$, and we evaluate
\begin{equation*}
  \sum_{p \leq y} \frac{g(p) \log p}{p} =
    \sum_{p \leq y} \frac{7\log p}{p^{3/2}} \ll 1.
\end{equation*}
Proposition~\ref{prop A3a} therefore applies with $\kappa=0$, yielding
$$
  \sum_{d \leq y} \frac{g(d)}d = c(g) + O_g((\log y)^{-1}) \ll 1.
$$
The proposition now follows from equation~\eqref {divisor trick 1} and this last estimate.
\end{proof}

\begin{lemma}
Define $S(y) = \sum_{b \le y} \rho(b)$. For all $y\ge2$, we have $S(y) = \frac6{\pi^2}y\log y+O(y)$.
\label{rho asymptotic lemma}
\end{lemma}

\begin{proof}
Define the multiplicative function $g$ by its values on prime powers as follows:
\[
g(p^\alpha) = \mu^2(p^\alpha) \text{ for $p$ odd or $\alpha\ge4$;} \qquad g(2)=0,\, g(4)=1,\, g(8)=2.
\]
As in the proof of the previous lemma, one can check that $\rho(m) = \sum_{d\mid m} g(d)$, and so
\begin{equation}
  S(y) = y\sum_{d\le y} \frac{g(d)}d + O\bigg( \sum_{d\le y} g(d) \bigg).
  \label{divisor trick 2}
\end{equation}
Since $g$ is bounded by 2, the error term is $O(y)$. We note that $g(p^\alpha) \le 2$ for all prime powers $p^\alpha$, and we evaluate
\begin{equation*}
  \sum_{p \leq y} \frac{g(p) \log p}{p} =
    \sum_{2<p \leq y} \frac{\log p}{p} = \log y + O(1).
\end{equation*}
Proposition~\ref{prop A3a} therefore applies with $\kappa=1$, yielding
\begin{align*}
  \sum_{d \leq y} \frac{g(d)}d &= \log y \bigg( 1-\frac12 \bigg) \bigg( 1+\frac02+\frac14+\frac28 \bigg) \prod_{p>2} \bigg( 1-\frac1p \bigg) \bigg( 1+\frac1p \bigg) + O_g(1) \\
  &= \log y \prod_p \bigg( 1-\frac1{p^2} \bigg) + O(1) = \frac{\log y}{\zeta(2)} + O(1) = \frac6{\pi^2}\log y+O(1).
\end{align*}
The proposition now follows from equation~\eqref {divisor trick 2} and this last asymptotic formula.
\end{proof}

When $\rho$ is generalized to $\rho_f$ for any reducible, nonsquare quadratic polynomial $f$ with integer coefficients, the proof of Lemma~\ref {rho asymptotic lemma} generalizes to yield
\begin{equation}
  \sum_{m\le y} \rho_f(m) = c(f) y\log y+O_f(y)
\label{rho_f asymptotic}
\end{equation}
for some positive constant $c(f)$. It is also easy to show (by splitting the sum into dyadic intervals, for example) that
\begin{equation}
\sum_{a \le A} \frac{\rho(a)}{a^{3/4}} \ll A^{1/4}\log A.
\label{partial summation equation}
\end{equation}

\medskip
{\em Acknowledgements.} We thank John Friedlander, Roger Heath-Brown, and Hugh Montgomery for their insights into the work of Erd\H os--Tur\'an and Hooley.


\begin{thebibliography}{99}

\bibitem{AS}
M.~Abramowitz and I.~A.~Stegun, {\em Handbook of Mathematical Functions}, Dover Publications Inc., New York (1965).

\bibitem{dujella}
A.~Dujella, ``On the number of Diophantine $m$-tuples'', {\em Ramanujan J.} {\bf 15} (2008), no.~1, 37--46.

\bibitem{other dujella}
A.~Dujella, ``There are only finitely many Diophantine quintuples'', {\em J.\ Reine Angew.\ Math.} {\bf 566}  (2004), 183--214.
 
\bibitem{hooley}
C.~Hooley, ``On the distribution of the roots of polynomial congruences'', {\em Mathematika} {\bf 11} (1964), 39--49.

\bibitem{other hooley}
C.~Hooley, ``On the number of divisors of quadratic congruences'', {\em Acta.\ Math.} {\bf 110} (1963), 97--114.

\bibitem{behemoth}
G.~Martin, ``An asymptotic formula for the number of smooth values of a polynomial'', {\em J.~Number Theory} {\bf 93} (2002), no.~2, 	108--182.

\bibitem{montgomery}
H.~L.~Montgomery, {\em Ten Lectures on the Interface Between Analytic Number Theory and Harmonic Analysis}, CBMS Regional Conference Series in Mathematics, 84, American Mathematical Society, Providence, RI (1994).

\bibitem{NZM}
I.~Niven, H.~S.~Zuckerman, and H.~L.~Montgomery, {\em An Introduction to the Theory of Numbers}, 5th edition, John Wiley \&\ Sons, Inc., New York (1991).

%\bibitem{ExamplePaper}
%E.~Edelman, ``The probability that a random real Gaussian matrix has $k$ real eigenvalues, related distributions, and the circular law'', {\em J.~Multivariate~Anal.} {\bf 60} (1997), no.~2, 202--232.

%\bibitem{ExampleBook}
%H.~L.~Montgomery and R.~C.~Vaughan, {\em Multiplicative Number Theory I: Classical Theory}, Cambridge University Press (2007).

\end{thebibliography}
\end{document}